\documentclass[11pt]{article}
\usepackage[T1]{fontenc}
\usepackage[utf8]{inputenc}
\usepackage{amsmath,amssymb,amsthm,mathtools}
\usepackage{geometry}
\usepackage{hyperref}
\newcommand{\N}{\mathbb{N}}
\newcommand{\Z}{\mathbb{Z}}

\newcommand{\R}{\mathbb{R}}
\newcommand{\C}{\mathbb{C}}
\newcommand{\HH}{\mathbb{H}}

\providecommand{\Baire}{\mathbb{N}^{\mathbb{N}}}
\providecommand{\dom}{\operatorname{dom}}
\providecommand{\tr}{\operatorname{tr}}
\providecommand{\rank}{\operatorname{rank}}
\providecommand{\im}{\operatorname{im}}
\providecommand{\diag}{\operatorname{diag}}
\providecommand{\Sym}{\operatorname{Sym}}
\providecommand{\Circ}{\operatorname{Circ}}
\providecommand{\Mult}{\rightrightarrows}
\providecommand{\GL}{\operatorname{GL}}
\providecommand{\End}{\operatorname{End}}

\newtheorem{theorem}{Theorem}[section]
\newtheorem{proposition}[theorem]{Proposition}
\newtheorem{lemma}[theorem]{Lemma}
\newtheorem{corollary}[theorem]{Corollary}
\theoremstyle{definition}\newtheorem{definition}[theorem]{Definition}
\theoremstyle{remark}
\theoremstyle{remark}\newtheorem{question}[theorem]{Question}
\theoremstyle{remark}\newtheorem{example}{Example}

\title{Uniform Computability of Isotypic Block Diagonalization for Finite-Group Representations\thanks{Preliminary results were presented at CCA 2026, Trier, Germany.}}
\author{Konrad Burnik\thanks{Independent researcher, Amsterdam, the Netherlands. Email: \texttt{kburnik@gmail.com}}}

\date{}

\begin{document}
\maketitle

\begin{abstract}
We study the universal isotypic block-diagonalization problem in the Type-2 theory of effectivity (TTE).
We obtain a uniform computability result for the universal block diagonalizer for commutants of given 
finite group representations. As an application, when the given representation is orthogonal, 
we compute ordinary $QR$ factorizations blockwise in the isotypic basis. After transforming 
back to the original coordinates this yields a factorization $A=QX$, where $Q$ is orthogonal 
and $X$ is generally not upper triangular. Moreover, given a Wedderburn-adapted orthogonal basis, 
the transformed-back factors can be computed within the commutant.
\end{abstract}



\section{Introduction}
\label{sec:intro}
We are motivated by the following known facts from classical linear algebra.
First, it is well known that any symmetric matrix can be diagonalized.
Thus, for all $n\in \N$ we have
\[
  \forall A \in \Sym_n(\R) \quad \exists U_A \in O(n) \quad U_A^T A U_A \mbox{ is a diagonal matrix},
\]
where $\Sym_n(\R)$ is the set of all real symmetric square matrices of order $n$ and $O(n)$ is the set of all orthogonal matrices of order $n$. Hence, for every symmetric matrix $A$ we can find an orthogonal matrix that diagonalizes it. This orthogonal matrix depends on $A$: the columns of $U_A$ are eigenvectors of $A$.
Second, it is well known that the class of complex circulant matrices is \emph{simultaneously diagonalized by the same change of basis}.
Hence, for all $n\in \N$ we have
\[
  \exists U \in U(n)\quad \forall A \in \Circ(n) \quad U^* AU \mbox{ is diagonal},
\]
where $\Circ(n)$ denotes the set of all circulant square matrices of
order $n$. Notice the reversal of the logical quantifiers: $\forall A \ \exists U_A$ versus $\exists U\ \forall A$.
Here, the matrix $U$ diagonalizes the whole class of (complex) circulant matrices. 
The matrix $U$ is the famous DFT matrix of order $n$. Crucially, the entries of this matrix $U$ 
do not depend on the entries of any of the circulant matrices. For the real version of the DFT, 
the best we can achieve is block-diagonalization: in that case $U^TAU$ is a block-diagonal 
matrix with $1\times1$ and $2\times2$ blocks. Thus, for all $n \in \N$ we have
\[
  \exists U \in O(n)\quad \forall A \in \Circ_\R(n)  \quad U^T A U \mbox{ is block diagonal}.
\]
Hence, there exists a \emph{universal} $U$ that block-diagonalizes all matrices from the class $\Circ_\R(n)$.

These observations motivate the following definition.
\begin{definition}
Let $\mathcal{A}$ be a class of $n \times n$ \textbf{real} matrices.
Let $k\geq 1$ and let $d_1,\dots,d_k\geq 1$ be integers such that $d_1+\dots+d_k=n$.
An invertible matrix $U$ is a \textbf{universal block diagonalizer} (UBD) for $\mathcal{A}$ with block sizes $(d_1,\dots,d_k)$ if $U^{-1} A U$ is block-diagonal with diagonal blocks of sizes $d_1,\dots,d_k$ for every $A \in \mathcal{A}$. If, moreover, $U \in O(n)$, we call $U$ an \textbf{orthogonal} universal block diagonalizer. In that case $U^{-1}AU = U^TAU$.
\end{definition}
The word \emph{universal} emphasizes that a single matrix $U$ works for all $A \in \mathcal{A}$.

\begin{example}[UBD for even-order centrosymmetric matrices]
  A $2n \times 2n$ matrix
  $A$ is centrosymmetric if and only if there exist $n \times n$ matrices
  $B$, $C$ such that
  \[
  A = \begin{pmatrix}
     B & C J \\
     JC & J B J
  \end{pmatrix},
 \]
 where $J$ is the $n\times n$ exchange matrix \cite{Weaver1985}.
 It is well known that the algebra of $2n \times  2n$ centrosymmetric matrices
 has an orthogonal universal block diagonalizer
 \[
    U = \frac{1}{\sqrt{2}}\begin{pmatrix}
        I & I \\
        J & -J
    \end{pmatrix},
 \]
 and $U^TAU = \diag(B+C,\, B-C)$.
 A similar example exists for centrosymmetric matrices of odd order.
\end{example}

An important obstruction in constructing a universal block diagonalizer is the
general non-computability of the rank. In the construction studied here, each block dimension is the rank of a corresponding projector. We will show that the finite-group structure allows these particular ranks to be recovered exactly.

We first explain how universal block diagonalization naturally stems from a specialized family of idempotents.

\begin{definition}[\cite{Koliha2001}]
    Let $P_i$, for $i\in\{1,\dots,k\}$, be such that
    \begin{enumerate}
        \item $P_i^2=P_i$ for all $i\in\{1,\dots,k\}$,
        \item $P_i P_j =0$ for all $i,j \in \{1,\dots, k\}$, $i \not = j$,
        \item $P_1 + \dots + P_k = I$.
    \end{enumerate}
    Then we say $\{P_1,\dots,P_k\}$ is a \emph{resolution of the unit matrix}.
\end{definition}

A resolution of the unit matrix block-diagonalizes its \emph{commutant}.
\begin{proposition}
  Let $\mathcal{P}=\{P_1,\dots,P_k\}$ be a resolution of the unit matrix.
  Let
  \[
    C(\mathcal{P}) = \{A:  A P_i = P_i A \text{ for all } i \}.
  \]
  Let $\mathcal{B}_i$ be a basis for $\im(P_i)$ for each $i=1,\dots,k$.
  Let $U$ be a matrix whose columns are the elements of $\mathcal{B}_1,\dots,\mathcal{B}_k$ in that order. Then
  $U$ is invertible and $U^{-1} A U$ is block-diagonal for each $A \in C(\mathcal{P})$.
  \label{prop:resolution}
\end{proposition}

\begin{proof}
  Every $x$ satisfies $x = \sum_i P_i x$, and this decomposition is unique since $P_iP_j=0$ for all $i\neq j$ and $P_i^2=P_i$. Hence $\R^n = \im P_1 \oplus \cdots \oplus \im P_k$, and $U$ is invertible. If $AP_i = P_iA$, then $A(\im P_i) = AP_i\R^n = P_iA\R^n \subseteq \im P_i$. Thus every $\im P_i$ is invariant under $A$, and $U^{-1}AU$ is block-diagonal.
\end{proof}

Hence, if we find a resolution of the unit matrix, it will automatically block-diagonalize all matrices that commute with it.

In order to find the resolution of the unit matrix, we use finite group representation theory to construct it in a systematic way.

This approach will also solve the problem of finding a universal block diagonalizer which is \emph{computable in the Type-2 setting}. In particular, we
focus on the variant of this problem for matrix classes that are precisely the \emph{commutants of representations of finite groups}.

The reason why we focus on these is that various matrix classes used in numerical linear algebra can naturally be defined as commutants.
Some particular examples are centrosymmetric
matrices, whose symmetry comes from commuting with a representation of the group $\Z /2 \Z$,
and circulant matrices, which commute with a representation of the group $\Z / \ell \Z$.
More generally, we study matrices that arise as the commutant of
a finite group representation.
Some closely related classes of matrices, like \emph{skew-centrosymmetric} matrices, are defined via the more general notion of an \emph{intertwiner of a representation of a finite group}, and can be block-diagonalized with a pair
of matrices $U$, $V$. This is out of scope for this paper.

To proceed with answering the question of computability of a universal block diagonalizer
for commutants, we must also observe the following
known obstruction in Type-2 computability theory. Namely, it is well known that the rank is not
computable from Cauchy names of the input matrix. This stems from the fact that rank is a discontinuous function.
Hence, the rank of the matrix can be provided as \emph{discrete advice} (\cite{Ziegler2012}).

For obtaining the block sizes, it is necessary to know the dimensions of the corresponding invariant subspaces. Equivalently, one needs the ranks of certain projectors. In general, such rank information may have to be supplied as discrete advice. In the finite-group setting studied here, however, this discrete information need not be supplied externally. We show that it can be recovered uniformly from the algebraic structure of the input. Namely, the central idempotents of the group algebra are mapped to projectors whose images are the isotypic components. The ranks of these projectors equal their traces. Since these traces are integers, they can be recovered exactly from sufficiently accurate approximations. Thus the finite-group structure supplies the discrete information required to compute the block dimensions.

In particular, we study representations of finite groups and matrix classes defined as
commutants of such representations.

The representation theory of finite groups provides formulas for the required block sizes.
By the representation theory of finite groups, the block sizes are known to be integers and
we can recover them by approximating and rounding the results of these formulas.

Our main theorem is the Type-2 computability of the universal block diagonalization of the following form.
The universal block diagonalizer for the isotypic decomposition of the commutant is \emph{uniformly computable
from the given finite group information}.

As an application, we use the computability of the (orthogonal) universal block diagonalizer
of the commutant to compute ordinary $QR$ factorizations blockwise in isotypic coordinates for matrices from the commutant of
a given representation of a finite group.

Our main contribution is the following. Universal block diagonalization can fail to be computable from Cauchy names alone. Already for the semisimple algebras generated by a single symmetric matrix, universal block diagonalizers are not computable from the generator alone (Corollary~\ref{cor:RA}), even when the target block sizes are fixed. In the finite-group construction studied here, an additional obstruction is that the isotypic block sizes are ranks of projectors. We show that the group data supplies exactly the discrete information needed to recover these particular ranks: the block sizes are traces of the isotypic projectors, hence integers that can be recovered by rounding. Consequently, the universal block diagonalizer is uniformly Type-2 computable from the Cayley table of the group and Cauchy names of the representing matrices (Theorems~\ref{thm:ubd-main} and~\ref{thm:ubd-orth}). The underlying formulas from representation theory are classical. What is new is their analysis in the Type-2 setting.

The rest of this paper is organized as follows. In Section \ref{sec:related-work-and-scope}
we discuss the related work regarding effective linear algebra and set the scope for this paper.
In Section \ref{sec:type-2} we review the basic notions of the Type-2 Theory of Effectivity (TTE)
that are needed to state and prove our results. In Section \ref{sec:fin-group-repr} we recall the
notions from the representation theory of finite groups that will enable us to construct the required
projectors. In Section \ref{sec:comp-univ-iso-red} we state and prove our main result. In Section
\ref{sec:qr-fact-after-univ-red} we apply our main result by computing ordinary $QR$ factorizations blockwise in isotypic coordinates and then transforming the factors back. In Section \ref{sec:examples} we specialize this construction to centrosymmetric matrices and relate it to the structure-preserving $QX$ factorization introduced in \cite{Burnik2015}.
We dedicate Section \ref{sec:open-questions} to open questions and potential future directions of
this work. In Section \ref{sec:conclusion} we give a conclusion to this work, emphasizing the most
important points.

\section{Related work and scope}
\label{sec:related-work-and-scope}
In the work of Ziegler and Brattka (\cite{ZieglerBrattka2004}), the authors study Type-2 computability of classical problems in linear algebra such as solving a linear system, finding bases of the image and the kernel of a matrix, and calculating an orthonormal basis for a given space. The Type-2 Turing machine is the model of computation used in computable analysis, which takes infinite input and provides infinite output. In calculating with real numbers, it does so with arbitrary approximations of reals. It is well known that checking whether a real number (given by a Cauchy name) is zero is not possible in the Type-2 model. Various problems in linear algebra require checking such an equality, e.g., calculating the rank of a matrix. However, by work of Ziegler (\cite{Ziegler2012}), when discrete advice is supplied as additional input to the machine, these problems become computable.
Exact algorithms for group representations and related algebraic decompositions are known: for irreducible representations \cite{BabaiRonyai1990, Dabbaghian2005}, for character tables \cite{Dixon1967, Unger2006}, and for Wedderburn decompositions \cite{Eberly1989, OlteanuDelRio2009, Bremner2011}. These work with exact arithmetic. Here, the novelty is the use of a Type-2 machine, which works only with approximations. The character and projector formulas are classical (\cite{Isaacs1976, Serre1977}).
We use these known formulas in the Type-2 setting, which uses the Cauchy-name input model, to obtain computability of the isotypic decomposition, since the classical formulas themselves yield exact values of the needed block sizes.
The isotypic decomposition itself is classical. Our contribution is its computability analysis in the Type-2 setting.
A recent preprint by Hsieh and Kang \cite{Kang2026} uses the same central idempotents to project onto an isotypic component and then constructs explicit irreducible representations via an optimization problem. However, it does not address Type-2 computability.
Given only the isotypic basis, our blockwise $QR$ construction yields a transformed-back factorization $A=QX$. We do not claim that both transformed-back factors are from the
commutant in general: $Q$ is orthogonal, while $U^TXU$ is block diagonal with upper-triangular diagonal blocks. Given a Wedderburn-adapted orthogonal basis as additional input, both transformed-back factors lie in the commutant (Corollary~\ref{cor:qr-wedderburn}).
Our computability result for the block diagonalizer is limited to the isotypic decomposition. Computing a Wedderburn-adapted basis from the group data is future work.

\section{Type-2 preliminaries and discrete information}
\label{sec:type-2}
Given an arbitrary Cauchy-named matrix, the rank is not computable \cite[Theorem~7]{ZieglerBrattka2004}. It is easy to show this already for matrices
of size $1\times1$. Suppose $A = ( x )$, where $x$ is given by a Cauchy name. Then the rank of $A$ depends on whether
$x = 0$. From a Cauchy name of $x$ it is not possible, in general, to decide whether $x = 0$.
On the other hand, the property $x \neq 0$ is semidecidable.
The reason the rank function is not computable is that the rank function is discontinuous, so checking whether a real number is exactly zero is not computable.
Since the rank cannot be calculated from Cauchy names of the matrix entries, neither the basis for the image nor that of the kernel can be calculated in general. In the following we give standard definitions and
facts from TTE (for a detailed introduction see \cite{Weihrauch2000}).

\begin{definition}
A \textbf{representation} of a set $X$ is a partial surjection
\[
  \delta_X : \subseteq \Baire \to X .
\]
If $\delta_X(p)=x$, then $p$ is a \textbf{$\delta_X$-name} of $x$.
The pair $(X,\delta_X)$ is a represented space.
\end{definition}

\begin{definition}
For represented spaces $(X,\delta_X)$ and $(Y,\delta_Y)$, a map
$f: \subseteq X \to Y$ is \textbf{computable} if some Type-2 machine
computes a realizer $F: \subseteq \Baire \to \Baire$ such that
\[
  \delta_Y(F(p)) = f(\delta_X(p))
\]
for all $p\in \dom(f\circ\delta_X)$.
\end{definition}

Linear-algebra tasks can have many valid answers,
so we use multi-valued maps
\[
  f: \subseteq X \Mult Y .
\]
A Type-2 machine computes $f$ if, from every name of $x\in\dom(f)$, it
outputs a name of some $y\in f(x)$.

\medskip
\begin{definition}[Cauchy representation]
A $\delta_\R$-name (Cauchy name) of $x\in\R$ is a sequence of rationals $(q_m)_{m\in\N}$
with
\[
  |q_m-x|\le 2^{-m}\qquad \forall m\in\N.
\]
Matrices in $M_n(\R)$ are named by \emph{interleaving}
the names of their individual entries, i.e., by the representation
$\delta_\R^{n \times n}$.
Finite data, such as $n$ or a
Cayley table, are given exactly as finite strings.
\end{definition}

\begin{lemma}
  Let $x$ be a real number given by its Cauchy name $(q_m)_{m\in\N}$. Suppose $x \in \Z$. Then there is
  a uniform algorithm that calculates $x$ exactly.
  \label{lem:int-rec}
\end{lemma}

\begin{proof}
  Let $k = 2$, so that $|x - q_k| \le 1/4 < 1/3$. There is nothing
  specific in the choice of the bound $1/3$. In fact, any error bound $< 1/2$ would work.
  Now run through $m \in \Z$ in the order $0, -1, 1, -2, 2, \dots$ until
  we find $m$ such that $|m - q_k| < 1/3$. Such an $m$
  exists, because $x \in \Z$ and $|x -  q_k| < 1/3$, so this search terminates.
  The integer $m$ found is equal to $x$, since by the triangle inequality
  we have $|x - m| \leq |x - q_k| + |q_k - m| < 2/3$,
  and since $x, m \in \Z$ this gives $x = m$. Hence, output $m$.
\end{proof}

\begin{lemma}
  Let $P$ be an $n \times n$ matrix such that $P^2=P$.
  Then $\rank(P)$ can be calculated exactly from the Cauchy name of $P$.
  \label{lem:idem-rank}
\end{lemma}

\begin{proof}
  For any idempotent $P$, we have $\R^n = \im P \oplus \ker P$.
  Hence, $P$ acts as the identity on $\im P$ and as $0$ on $\ker P$.
  Thus $P$ is diagonalizable with eigenvalues in the set $\{0,1\}$, and
  $\rank(P) = \tr(P)$.
  On the other hand, the trace $\tr(P)$ is computable from
  the Cauchy name of $P$. Further, since we know that $\tr(P) \in \{0,\dots,n\}$,
  we can use integer recovery (Lemma \ref{lem:int-rec}) to obtain $\tr P$, and hence $\rank(P)$, exactly.
  Note that here we do not assume that $P$ is symmetric or orthogonal.
\end{proof}

\begin{lemma}[see {\cite[Theorem~11]{ZieglerBrattka2004}}]
  Let $A$ be an $n \times n$ matrix given by Cauchy names.
  Let $r \in \N$ be the rank of $A$, supplied as an exact integer input.
  Then we can uniformly compute a basis for $\im A$ from this input.
  \label{lem:im-ker-computable-from-given-rank}
\end{lemma}

\begin{proof}
  Assume $r = \rank A$ is supplied as input. If $r=0$, then $\im A=\{0\}$ and the empty list is a basis, so suppose $r>0$.
  The following algorithm terminates and calculates a basis for $\im A$. We proceed by enumerating
  all $r \times r$ minors of $A$. Equivalently, we search for $r$ linearly independent columns of $A$.
  For each chosen $r\times r$ minor $M$ we can semidecide whether $\det M \not =0$, since we can
  semidecide whether a given real number is nonzero (see \cite[Lemma~10(a)]{ZieglerBrattka2004}). We test the minors in a dovetailing
  fashion. Since $\rank A = r$, some $r \times r$ minor is nonzero, so this procedure eventually halts on a minor $M$ with $\det M \not = 0$.
  The $r$ columns of $A$ indexed by the columns of $M$ are then linearly independent, and since $\rank A = r$, they form a basis of $\im A$.
\end{proof}

\subsection{Some known computable linear algebra results}
Already for a single symmetric matrix, finding an orthogonal matrix that diagonalizes it is not computable.
\begin{proposition}[M.~Ziegler and V.~Brattka {\cite[Example~18 and Section~3.5]{ZieglerBrattka2004}}]
 Let $n \geq 2$. The multi-valued eigenbasis problem
  \[
    \Sym_n \Mult O(n),
     \qquad A \mapsto \{U \in O(n): U^{T} A U\text{ is diagonal}\}
  \]
  is \textbf{not} Type-2 computable.
  \label{prop:eigenbasis}
\end{proposition}
The argument in \cite[Example~18]{ZieglerBrattka2004} uses Rellich's family $B(\epsilon)$ of symmetric $2\times 2$ matrices, whose eigenvectors depend discontinuously on $\epsilon$ at $\epsilon = 0$. For $n > 2$, the same argument applies to the matrices obtained by embedding $B(\epsilon)$ in the upper-left $2\times 2$ block of $I_n$.

For a symmetric matrix $A$, let $\R[A]$ denote the algebra of all real polynomials in $A$. Since $A$ is orthogonally diagonalizable, $\R[A] \cong \R^{|\sigma(A)|}$ is semisimple.

\begin{corollary}
  Let $n \geq 2$. The multi-valued map which sends $A \in \Sym_n$ to the set of orthogonal universal block diagonalizers for $\R[A]$ with block sizes $(1,\dots,1)$ is \textbf{not} Type-2 computable.
  \label{cor:RA}
\end{corollary}

\begin{proof}
A matrix $U \in O(n)$ is an orthogonal universal block diagonalizer for $\R[A]$ with block sizes $(1,\dots,1)$ if and only if $U^TAU$ is diagonal: necessity holds since $A \in \R[A]$, and sufficiency holds since $U^T f(A) U = f(U^TAU)$ is diagonal for every polynomial $f$. Hence the map coincides with the eigenbasis problem, which is not computable by Proposition~\ref{prop:eigenbasis}.
\end{proof}

Thus, for the semisimple algebras $\R[A]$, universal block diagonalizers are not computable from the generator $A$ alone, even though the block sizes $(1,\dots,1)$ are known in advance. However, given $|\sigma(A)|$, the number of distinct eigenvalues of $A$, as \emph{discrete advice} on the input, the problem becomes computable:
\begin{proposition}[M.~Ziegler and V.~Brattka {\cite[Theorem~19]{ZieglerBrattka2004}}]
 The advised multi-valued problem
  \[
     (A, |\sigma(A)|) \mapsto \{U \in O(n): U^{T} A U\text{ is diagonal}\}
  \]
  is Type-2 computable.
\end{proposition}
These results illustrate the general pattern that motivates this paper: a problem which is not computable from Cauchy names alone becomes computable once suitable discrete advice is supplied. In our setting, the corresponding discrete information consists of the block sizes $d_i$. Unlike externally supplied discrete advice, these integers are computed uniformly from the finite-group representation data (Section~\ref{sec:comp-univ-iso-red}).

\section{Finite-group representations and real character data}
\label{sec:fin-group-repr}
In this section we give an overview of known results from the representation theory of finite groups.
These results are essential for constructing the resolution of the unit matrix that will circumvent
the non-computability of the rank. For a more comprehensive introduction the reader can consult the
classic texts (see \cite{Isaacs1976}, \cite{Serre1977}).

Let $G$ be a finite group. Let $V$ be a finite-dimensional real vector space.
A \emph{representation} of $G$ in $V$ is
a homomorphism $G \to \GL(V)$.\footnote{Here the term \emph{representation} stems from the theory of finite groups and is unrelated to representations in the sense of TTE (Section~\ref{sec:type-2}), which we denote by $\delta$.}
Since $V$ is finite-dimensional, after choosing a basis we identify $V$ with $\R^n$ and write $\rho: G \to \GL_n(\R)$ instead. A linear subspace $W \subseteq V$ is an \emph{invariant} subspace for $\rho$ if $\rho(g) W \subseteq W$ for all $g \in G$.

We say that $\rho$ is \emph{irreducible} if the only invariant subspaces for $\rho$ are $\{0\}$ and $V$.
Let $\rho_1,\dots,\rho_k$ denote a complete list of pairwise non-equivalent real irreducible representations of $G$. For each $i$, choose a representation space $W_i$ so that
\[
  \rho_i:G\to \GL(W_i),
\]
and put $n_i=\dim_\R W_i$.

By $\chi_\rho: G \to \R$, $\chi_\rho(g) = \tr(\rho(g))$ for all $g \in G$, we denote the \emph{character} of $\rho$.
By $\chi_i$ we denote the character of the irreducible representation $\rho_i$.
Since the eigenvalues of $\rho(g)$ are roots of unity and $\chi_\rho$ is real-valued, we have $\chi_\rho(g^{-1}) = \overline{\chi_\rho(g)} = \chi_\rho(g)$ for all $g \in G$.

For every representation $\rho$ we have a direct sum decomposition
\[
   \rho \cong \bigoplus_{i=1}^k\bigoplus_{j=1}^{m_i}\rho_i,
\]
where $m_i\in\N$ is the multiplicity of $\rho_i$ in $\rho$. For the characters this gives
\[
   \chi_\rho = m_1 \chi_1 + \cdots + m_k \chi_k .
\]
Let
\[
  D_i=\End_G(\rho_i)
  =\{T\in\End_\R(W_i):T\rho_i(g)=\rho_i(g)T\text{ for all }g\in G\}.
\]
Thus $D_i$ is the commutant of the irreducible representation $\rho_i$. By the real Schur lemma, $D_i$ is isomorphic to $\R$, $\C$ or $\HH$, and
$\kappa_i=\dim_\R D_i\in \{1,2,4\}$.

Define
\[
  \operatorname{Hom}_G(\rho_i,\rho)
  =\{T:W_i\to \R^n\text{ linear}:T\rho_i(g)=\rho(g)T\text{ for all }g\in G\}.
\]
The \emph{$i$-th isotypic component} of $\rho$ is the subspace
\[
  V_i=\operatorname{span}\{T(w):T\in\operatorname{Hom}_G(\rho_i,\rho),\ w\in W_i\}.
\]
Equivalently, $V_i$ is the sum of all $\rho$-invariant subspaces on which the restricted representation is isomorphic to $\rho_i$. If $m_i$ is the multiplicity of $\rho_i$ in $\rho$, then
\[
  \rho|_{V_i}\cong \bigoplus_{j=1}^{m_i}\rho_i.
\]
After omitting the indices with $m_i=0$, the decomposition
\[
  \R^n=\bigoplus_i V_i
\]
is called the \emph{real isotypic decomposition} of $\rho$.

The set
\begin{equation}
\mathcal{C}(G, \rho) = \End_G (\rho) = \{ A \in M_n(\mathbb{R}) \mid A \rho(g) = \rho(g) A, \mbox{ for all } g \in G \}
\label{eq:commutant}
\end{equation}
is called the \emph{commutant} of $G$ with respect to $\rho$.

In the following, we calculate explicitly the multiplicities $m_i$ for a given representation $\rho$
by using the orthogonality relations for characters. Let
\[
  \langle \varphi, \psi \rangle = \frac{1}{|G|}\sum_{g \in G} \varphi(g)\psi(g^{-1}).
\]
For the real irreducible characters, the orthogonality relations read
\[
  \langle \chi_i,\chi_j\rangle
  =\kappa_i\delta_{i,j},
\]
where $\delta_{i,j}$ is the Kronecker delta. More generally, if
\[
  \sigma:G\to\GL(V),
  \qquad
  \tau:G\to\GL(W)
\]
are finite-dimensional real representations, then
\[
  \langle\chi_\sigma,\chi_\tau\rangle
  =
  \dim_\R\operatorname{Hom}_G(\tau,\sigma).
\]
Hence, for irreducible real representations $\rho_i$ and $\rho_j$,
\[
  \langle\chi_i,\chi_j\rangle
  =
  \dim_\R\operatorname{Hom}_G(\rho_j,\rho_i)
  =
  \kappa_i\delta_{i,j},
\]
by the real Schur lemma, where $\kappa_i=\dim_\R\End_G(\rho_i)$ (see \cite{Serre1977}).
Applying $\langle \cdot, \chi_i \rangle$
to $\chi_\rho = \sum_j m_j\chi_j$ for each $i \in \{1,\dots,k\}$, we obtain
$\langle \chi_\rho, \chi_i\rangle = m_i\kappa_i$, i.e.,
\[
  m_i= \frac{1}{\kappa_i|G|}\sum_{g\in G} \chi_\rho(g)\chi_i(g^{-1}).
\]
Since $m_i$ is the multiplicity counting the number of copies of $\rho_i$ in $\rho$,
$m_i$ is a non-negative integer. In the real case, we require $\kappa_i$ as a
corrective factor due to the different Schur types of $\rho_i$. Note also that $\kappa_i = \langle \chi_i, \chi_i \rangle$ is computable exactly from the character $\chi_i$.

Here we show that the multiplicities $m_i$ can be recovered exactly by a Type-2 machine.
Let $S$ be a finite set of generators for $G$. Given $\rho(g)$ for $g \in S$, we can calculate
$\rho(g)$ for every $g \in G$ by forming products of the $\rho(g)$ for $g \in S$.
Suppose now that we have $\rho(g)$ for each $g\in G$.
We proceed by calculating the character using the defining trace formula
$\chi_\rho(g)=\tr(\rho(g))$ for every $g \in G$ from the Cauchy
names for $\rho(g)$.
By the formula
\begin{equation}
  m_i= \frac{1}{\kappa_i|G|}\sum_{g\in G} \chi_\rho(g)\chi_i(g^{-1})
  \label{eq:mult-formula}
\end{equation}
we can calculate $m_i$ by using Cauchy names for $\chi_\rho$ and $\chi_i$.
Since we know that $m_i \in \N$ for each $i \in \{1,\dots,k\}$, we can use
integer recovery (Lemma \ref{lem:int-rec}) to output $m_i$ exactly.

\section{Computable universal isotypic block diagonalization}
\label{sec:comp-univ-iso-red}
In this section we define precisely the input for our Type-2 machine calculating the universal block diagonalizer.
Let $\mathcal X$ be the represented space of tuples consisting of a Cayley table for a finite group $G$, a generating set $S\subseteq G$, a dimension $n$, and a real matrix $A_s\in M_n(\R)$ for each $s\in S$. The finite components are represented by any fixed computable coding of finite strings. Each matrix $A_s$ is represented by $\delta_\R^{n\times n}$, and the finitely many components are combined by the usual computable finite-product coding.

We restrict $\mathcal X$ to the subspace $\operatorname{FinRep}$ consisting of those tuples for which the matrices $A_s$, $s\in S$, extend to a representation
\[
  \rho:G\to \GL_n(\R).
\]
The space $\operatorname{FinRep}$ carries the subspace representation inherited from $\mathcal X$. Membership in $\operatorname{FinRep}$ need not be decidable from the Cauchy names, and no such decision procedure is required for computability on this represented domain.

For any input from $\operatorname{FinRep}$ and any group element $g$, the matrix $\rho(g)$ can be calculated via finite products by performing a breadth-first search on the Cayley graph of $G$ with respect to $S$. We do not consider approximate representations here (see, e.g., \cite{BabaiFriedl1991, MooreRussell2015}).

The output is represented by a finite code for the number of nonzero blocks and their dimensions, together with a $\delta_\R^{n\times n}$-name of the matrix $U$. These finite and infinite components are combined by a fixed computable product coding.

We can now state the main theorem about computability of universal block diagonalization.
\begin{theorem}
  There exists a Type-2 machine which, given the input described above, outputs an invertible matrix $U$ and integers $d_1,\dots,d_r \geq 1$ with the following \emph{universal property}:
  the conjugated matrix $U^{-1}AU$ is block diagonal with diagonal blocks of sizes $(d_1,\dots,d_r)$ for every $A$ from the commutant $\End_G(\rho)$.
  Here $d_1 + \dots + d_r = n$, where $n$ is the dimension of the input representation,
  and each block corresponds to a nonzero real isotypic component of $\rho$.
  Hence, the same $U$ works for every $A$ in the commutant, and $U$ is uniformly computable from the input.
  In case $\rho$ is an orthogonal representation, $U$ can be chosen orthogonal (Theorem~\ref{thm:ubd-orth}).
 \label{thm:ubd-main}
\end{theorem}

For each irreducible real representation $\rho_i$, let
\[
  e_i
  =
  \frac{n_i}{\kappa_i|G|}
  \sum_{g\in G}\chi_i(g^{-1})\,g
  \in\R[G].
\]
These are the primitive central idempotents of $\R[G]$ corresponding to the irreducible real representation types. They satisfy
\[
  e_i^2=e_i,
  \qquad
  e_ie_j=0 \quad (\mbox{for all } i\neq j),
  \qquad
  \sum_{i=1}^k e_i=1.
\]

Let
\[
  \widehat\rho:\R[G]\to M_n(\R)
\]
be the linear extension of $\rho$. Define
\begin{equation}
  P_i=\widehat\rho(e_i)
  =\frac{n_i}{\kappa_i|G|}\sum_{g \in G} \chi_i(g^{-1}) \rho(g)
  \label{eq:projector}
\end{equation}
for $i = 1,\dots, k$.

\begin{proposition}
 Let $P_i$ be given by \eqref{eq:projector} for each $i=1,\dots,k$.
 We have the following:
 \begin{enumerate}
   \item[(i)] $\{P_1,\dots,P_k\}$ is a resolution of the unit matrix;
   \item[(ii)] the image $\im(P_i)=V_i$, where $V_i$ is the $i$-th isotypic component;
   \item[(iii)] each $P_i$ is Type-2 computable from the input.
 \end{enumerate}
 \label{prop:projectors}
\end{proposition}

\begin{proof}
 Since $\widehat\rho$ is an algebra homomorphism and
\[
  e_i^2=e_i,\qquad e_ie_j=0\ (\mbox{for all} i\neq j),\qquad \sum_i e_i=1,
\]
we obtain
\[
  P_i^2=P_i,\qquad
  P_iP_j=0\ (\mbox{for all }i\neq j),\qquad
  \sum_iP_i=I.
\]
Thus $\{P_1,\dots,P_k\}$ is a resolution of the unit matrix.

Moreover, the central idempotent $e_i$ acts as the identity on every irreducible subrepresentation isomorphic to $\rho_i$ and as zero on every irreducible subrepresentation not isomorphic to $\rho_i$. Consequently,
\[
  \im(P_i)=V_i,
\]
where $V_i$ is the $i$-th isotypic component, for each
$i \in \{1,\dots,k\}$.

Over $\C$, the projector onto the isotypic component of a complex irreducible character $\psi$ is $\frac{\psi(1)}{|G|}\sum_{g}\overline{\psi(g)}\rho(g)$ \cite{Serre1977}. Each real irreducible character is of the form $\chi_i = \psi$, $\chi_i = \psi + \bar\psi$ with $\psi \neq \bar\psi$, or $\chi_i = 2\psi$, according to whether $\kappa_i = 1$, $2$ or $4$ \cite{Serre1977}. Summing the complex projectors over the distinct complex constituents $\psi$ of $\chi_i$ gives \eqref{eq:projector} in each of the three cases.

 The property (iii) follows from the definition of $P_i$ and the following facts. All matrices $\rho(g)$ have computable
 Cauchy names. The complex character table of $G$ can be computed exactly from the Cayley table, with values in a cyclotomic field \cite{Dixon1967, Unger2006}. The Frobenius--Schur indicator of each \emph{complex} irreducible character $\psi$ then determines whether the corresponding real irreducible character is $\psi$, $\psi+\bar\psi$ or $2\psi$ \cite{Serre1977}. In the case of indicator $0$, the pair $\{\psi,\bar\psi\}$ determines a single irreducible real character, so we retain one representative from each complex-conjugate pair. All of these operations are performed exactly in the cyclotomic field before the resulting real character values are converted to Cauchy names. Hence the real irreducible characters $\chi_i$ are computable exactly, and their values $\chi_i(g^{-1})$ have computable Cauchy names.
 The value $|G|$ is read off from the Cayley table, $n_i = \chi_i(e)$, and
 $\kappa_i = \langle \chi_i, \chi_i \rangle$. All of these are computed exactly.
\end{proof}

In this section, we show that a basis for each nonzero $\im(P_i)$ is uniformly computable.
Let us define $r_i = \rank(P_i)$. From the previous discussion, we can calculate $m_i$
by the formula \eqref{eq:mult-formula}. Hence $r_i = m_i n_i$. By Lemma \ref{lem:im-ker-computable-from-given-rank}, for each $r_i$ we can calculate a basis
for $\im P_i$. Alternatively, since each $P_i$ is idempotent, we can use the trace formula $\rank(P_i)=\tr(P_i)$.
Hence, we can approximate $\tr(P_i)$ to within any precision $<1/2$ and then apply
the integer rounding step.
By invoking Lemma \ref{lem:im-ker-computable-from-given-rank} we can calculate a basis for each $\im P_i$. For generic Cauchy-named matrices, the known rank information enables us to calculate a basis of the image.

We now construct the universal block-diagonalizer $U$.
Let $\mathcal{B}_i$ be the computable basis obtained for the $i$-th isotypic component $V_i$.
Here, we only take the nonzero $V_i$. Since the $P_i$ form a resolution of the unit matrix, we have the direct sum decomposition $\R^n = \bigoplus_i V_i$.

Let $U$ be the matrix obtained by concatenating all the bases $\mathcal{B}_i$ into
a single matrix. Hence, the columns of $U$ form a basis for the whole space $\R^n$.
Thus $\rank U=n$, so $U$ is invertible. Further, $U$ is computable since all the
bases $\mathcal{B}_i$ are computable.

Next, we prove that every matrix $A$ from the commutant preserves every isotypic component, so $U$ block-diagonalizes
the whole commutant.
Let $A\in \End_G(\rho)$. We need to prove that $AP_i = P_iA$ for every $i \in \{1,\dots,k\}$.
Since $A$ is from the commutant, we have $A \rho(g) = \rho(g) A$ for all $g \in G$.
Since $P_i$ is a linear combination of the matrices $\rho(g)$, $g \in G$, it follows that
\[
  A P_i = P_i A
\]
for all $i \in \{1,\dots,k\}$.
Since $P_i$ and $A$ commute, it follows that $A(\im(P_i))\subseteq\im(P_i)$, so the image $\im P_i$ is an invariant subspace for every $A$.
Since $U$ is comprised of the bases $\mathcal{B}_i$ for $\im (P_i)$,
we conclude by Proposition~\ref{prop:resolution} that $U^{-1} A U$ is block-diagonal for every $A$, so $U$ is a universal block diagonalizer for $\End_G(\rho)$.

We now show that when we have an orthogonal representation $\rho$, namely $\rho(G) \subseteq O(n)$,
we can choose the universal block diagonalizer to be orthogonal.

We say that $P$ is an \emph{orthogonal projector} onto a subspace $W$ if $Px \in W$ and $x - Px \in W^{\bot}$ for all $x$.
Equivalently, $P$ is an orthogonal projector if and only if $P^2=P$ and $P^T=P$.

Since the characters are real-valued, we have $\chi_i(g^{-1}) = \chi_i(g)$, and since $\rho$ is orthogonal, we have $\rho(g)^T=\rho(g^{-1})$.
By reindexing the sum in \eqref{eq:projector} with $g \mapsto g^{-1}$ we
have $P_i^T=P_i$. Together with $P_i^2=P_i$ it follows
that $P_i$ is an orthogonal projector. Moreover, symmetric idempotents with $P_iP_j=0$ for all $i\neq j$ have mutually orthogonal images, so the isotypic components $V_i$ are mutually orthogonal.
We now take for every $P_i$ with $\im P_i \neq 0$ a basis of the image of $P_i$
and apply the Gram--Schmidt orthogonalization procedure to obtain an orthonormal basis.
Combining these bases into a matrix $U$, we obtain an orthogonal matrix such that $U^T A U$ is block diagonal for every $A$ from
the commutant.

  \begin{theorem}[UBD for finite group commutants is computable]
  Consider the multi-valued map which takes as input the Cayley table of a finite group $G$, a generating set $S \subseteq G$, the dimension $n$, and $\delta_\R^{n\times n}$-names of the matrices $\rho(s)$, $s \in S$, where the input belongs to the represented subspace of $\operatorname{FinRep}$ consisting of orthogonal representations $\rho:G\to O(n)$, and which returns the set of all $U \in O(n)$ whose consecutive column blocks are orthonormal bases of the nonzero isotypic components of $\rho$ (in the order of the computed characters $\chi_1,\dots,\chi_k$).
  This map is \textbf{uniformly} Type-2 computable. In particular, $U^T A  U$ is block-diagonal for each
  $A \in \mathcal{C}(G, \rho)$.
  \label{thm:ubd-orth}
  \end{theorem}

\begin{proof}
  On input $(G, S, n, \rho)$, for each $i=1,\dots,k$:
  \begin{enumerate}
      \item compute the real irreducible character $\chi_i$ (Proposition~\ref{prop:projectors});
      \item compute the projector
      \[
        P_i = \frac{n_i}{\kappa_i |G|}\sum_{g\in G} \chi_i(g^{-1}) \rho(g),
        \qquad \text{with } \kappa_i = \langle \chi_i, \chi_i \rangle;
      \]
      \item compute $d_i = \rank P_i = \tr(P_i) = n_i m_i$ exactly, by approximating $\tr(P_i)$ from the Cauchy names of $\rho(g)$ to within an error $<1/2$ and rounding to the nearest integer. If $d_i = 0$, skip this $i$;
      \item using Lemma~\ref{lem:im-ker-computable-from-given-rank} with the known rank $d_i$, select $d_i$ linearly independent columns of $P_i$, which form a basis of $V_i = \im P_i$, and apply Gram--Schmidt orthogonalization to them to obtain an orthonormal basis $\mathcal{B}_i$ of $V_i$. Gram--Schmidt is computable here, since the selected columns are linearly independent, so all norms by which we divide are nonzero.
  \end{enumerate}
  Return $U = [\mathcal{B}_{i_1} | \dots | \mathcal{B}_{i_r}]$, where $i_1 < \dots < i_r$ are the indices with $d_i \neq 0$. By the discussion above, $U$ is orthogonal and $U^TAU$ is block-diagonal for every $A \in \mathcal{C}(G,\rho)$.
\end{proof}

\begin{example}

The group is $\Z/2\Z = \{0, 1\}$ with the Cayley table:

\medskip
\begin{center}
\begin{tabular}{c|cc}
  $+_2$ & $0$ & $1$ \\
\hline
  $0$ & $0$ & $1$ \\
  $1$ & $1$ & $0$
\end{tabular}
\end{center}
\medskip

\noindent

Let $n = 4$. Let
$\rho \colon \Z/2\Z \to \mathrm{GL}_4(\R)$,
\[
  \rho(0) = I_4 =
  \begin{pmatrix}1&0&0&0\\0&1&0&0\\0&0&1&0\\0&0&0&1\end{pmatrix},
  \qquad
  \rho(1) = J_4 =
  \begin{pmatrix}0&0&0&1\\0&0&1&0\\0&1&0&0\\1&0&0&0\end{pmatrix}.
\]

  The group $\Z/2\Z$ has two conjugacy classes $\{0\}$ and $\{1\}$, 
  and two irreducible representations, both of degree~$1$ and of real type ($\kappa_1 = \kappa_2 = 1$):

\medskip
\begin{center}
\begin{tabular}{c|cc}
  & $0$ & $1$ \\
\hline
$\chi_1$  & $1$ & $1$ \\
$\chi_2$  & $1$ & $-1$
\end{tabular}
\end{center}
\medskip

\noindent
The degrees are $n_1 = \chi_1(0) = 1$ and $n_2 = \chi_2(0) = 1$.

For each $\rho(g)$, the algorithm approximates the
traces:
\[
  \chi_\rho(0) = \mathrm{tr}(I_4) = 4, \qquad
  \chi_\rho(1) = \mathrm{tr}(J_4) = 0.
\]

Since $\kappa_i = 1$ and $g^{-1} = g$ for all $g \in \Z/2\Z$, formula \eqref{eq:mult-formula} becomes $m_i = \frac{1}{|G|} \sum_{g \in G}
\chi_\rho(g)\, \chi_i(g)$, with $|G| = 2$:

\medskip
\noindent
\textbf{$\chi_1$:}
\[
  m_1 = \frac{1}{2}\bigl(\chi_\rho(0) \cdot \chi_1(0)
        + \chi_\rho(1) \cdot \chi_1(1)\bigr)
      = \frac{1}{2}(4 \cdot 1 + 0 \cdot 1) = 2.
\]

\noindent
\textbf{$\chi_2$:}
\[
  m_2 = \frac{1}{2}\bigl(\chi_\rho(0) \cdot \chi_2(0)
        + \chi_\rho(1) \cdot \chi_2(1)\bigr)
      = \frac{1}{2}(4 \cdot 1 + 0 \cdot (-1)) = 2.
\]

\medskip
\noindent
In the Type-2 setting, these sums are only approximated, but once the
error is less than $1/2$, rounding recovers the exact integers
$m_1 = 2$, $m_2 = 2$.

For the same reasons, formula \eqref{eq:projector} becomes $P_i = \frac{n_i}{|G|} \sum_{g \in G} \chi_i(g)\, \rho(g)$:

\medskip
\noindent
\textbf{Projection $P_1$} ($n_1 = 1$, $\chi_1(0) = 1$,
$\chi_1(1) = 1$):
\[
  P_1 = \frac{1}{2}\bigl(\rho(0) + \rho(1)\bigr)
      = \frac{1}{2}(I_4 + J_4)
      = \frac{1}{2}
        \begin{pmatrix}
          1 & 0 & 0 & 1 \\
          0 & 1 & 1 & 0 \\
          0 & 1 & 1 & 0 \\
          1 & 0 & 0 & 1
        \end{pmatrix}.
\]
$\operatorname{Im} P_1 = \{x \in \R^4 : J_4 x = x\}$, confirming
$d_1 = \mathrm{rank}(P_1) = 2$.

\medskip
\noindent
\textbf{Projection $P_2$} ($n_2 = 1$, $\chi_2(0) = 1$,
$\chi_2(1) = -1$):
\[
  P_2 = \frac{1}{2}\bigl(\rho(0) - \rho(1)\bigr)
      = \frac{1}{2}(I_4 - J_4)
      = \frac{1}{2}
        \begin{pmatrix}
          1 & 0 & 0 & -1 \\
          0 & 1 & -1 & 0 \\
          0 & -1 & 1 & 0 \\
          -1 & 0 & 0 & 1
        \end{pmatrix}.
\]
$\operatorname{Im} P_2  = \{x \in \R^4 : J_4 x = -x\}$, confirming
$d_2 = \mathrm{rank}(P_2) = 2$.

The universal block diagonalizer is
\[
  U =\frac{1}{\sqrt{2}}
      \begin{pmatrix}
        1 & 0 & 1 & 0 \\
        0 & 1 & 0 & 1 \\
        0 & 1 & 0 & -1 \\
        1 & 0 & -1 & 0
      \end{pmatrix}.
\]
\end{example}

\section{Blockwise QR in isotypic coordinates}
\label{sec:qr-fact-after-univ-red}
Let $\rho$ be orthogonal and let $A$ be a nonsingular $n \times n$ matrix such that $A \in \End_G(\rho)$.
Let $U$ be the orthogonal universal block diagonalizer for $\End_G(\rho)$ from Theorem~\ref{thm:ubd-orth}.

We use the following terminology in this section. A factorization
\[
  A=QX
\]
will be called a \emph{$QX$-type factorization relative to $U$} if $Q$ is orthogonal and $U^TXU$ is block diagonal with upper-triangular diagonal blocks. This terminology is motivated by the structure-preserving factorization for centrosymmetric matrices introduced by Burnik~\cite{Burnik2015}. For the standard centrosymmetric block-diagonalizer, the construction below recovers that factorization (see Section~\ref{sec:examples}).

\begin{theorem}
  Given the input of Theorem~\ref{thm:ubd-orth} together with a $\delta_\R^{n\times n}$-name of a nonsingular matrix $A \in \mathcal{C}(G,\rho)$, one can uniformly compute $U$ as in Theorem~\ref{thm:ubd-orth} and matrices $Q$, $X$ such that $A = QX$, $Q$ is orthogonal, and
  \[
    U^TXU = \diag(R_1,\dots,R_r),
  \]
  where each $R_j$ is upper triangular with positive diagonal. Equivalently, ordinary $QR$ factorizations are computed blockwise after passing to the isotypic coordinates determined by $U$.
  \label{thm:qr}
\end{theorem}

The following procedure uniformly computes this factorization.
\begin{enumerate}
  \item Compute $U$ and let $U^T A U=\diag(A_1,\dots,A_r)$.
  \item Compute the ordinary $QR$ factorizations $A_j=Q_j R_j$, with $R_j$ upper triangular with positive diagonal.
  \item Let $Q=U \diag(Q_1,\dots,Q_r) U^T$ and $X=U \diag(R_1,\dots,R_r) U^T$.
\end{enumerate}

\begin{proof}
  Each block $A_j$ is nonsingular, since $A$ is. The $QR$ factorization of a nonsingular matrix is computable by Gram--Schmidt orthogonalization of its columns: the columns are linearly independent, so all norms by which we divide are nonzero reals. Then $Q$ is orthogonal as a product of orthogonal matrices, and
  $QX = U\diag(Q_1R_1,\dots,Q_rR_r)U^T = U(U^TAU)U^T = A$.
\end{proof}

Note that $X$ itself is, in general, not upper triangular in the standard basis. Only $U^TXU$ is block diagonal with upper-triangular diagonal blocks. Thus, $A=QX$ is a $QX$-type factorization relative to $U$ in the sense defined above, while the ordinary $QR$ factorizations occur in the isotypic coordinates.
Further, we do not claim that the matrices $Q$ and $X$ lie in the commutant $\End_G(\rho)$ in general. In the following, we
give a condition under which the transformed-back factors $Q$ and $X$ are also in the commutant. In Corollary~\ref{cor:qr-wedderburn} below we show that both transformed-back factors lie in the commutant in general when, instead of an isotypic basis, a Wedderburn-adapted orthogonal basis is given. Computing such a basis from the group data is left for future work.

Here we specify one sufficient condition for the factors $Q$ and $X$ to lie in the commutant.
\begin{proposition}
  Let $\rho$ be orthogonal and let $A \in \End_G(\rho)$ be nonsingular. If every occurring multiplicity satisfies $m_i=1$, then the factors $Q$ and $X$ of Theorem~\ref{thm:qr} lie in the commutant $\End_G(\rho)$. Moreover, $U^TXU$ is diagonal.
  \label{prop:all-mi-equal-one}
\end{proposition}

\begin{proof}
Let $U_j$ denote the column block of $U$ spanning the isotypic component $V_j$, and let $\rho_j'(g) = U_j^T\rho(g)U_j$. Since $V_j$ is invariant under every $\rho(g)$, $\rho_j'$ is an orthogonal representation of $G$ on $\R^{d_j}$, and $\rho(g) = U\diag(\rho_1'(g),\dots,\rho_r'(g))U^T$. Since $m_j = 1$, the representation $\rho_j'$ is irreducible.
Let $A_j = U_j^TAU_j$. Since $V_j$ is invariant under $A$ and under every $\rho(g)$, and $A$ commutes with $\rho(g)$, the matrix $A_j$ commutes with every $\rho_j'(g)$, i.e., $A_j \in D := \End_G(\rho_j')$.
By the real Schur lemma, $D$ is a division algebra: every nonzero element of $D$ is invertible.
Since $\rho_j'$ is orthogonal, $D$ is closed under transposition, so $A_j^TA_j \in D$, and $A_j^TA_j$ is symmetric.
Let $\lambda$ be a real eigenvalue of $A_j^TA_j$. Then $A_j^TA_j - \lambda I \in D$ is singular, hence zero, so
\[
  A_j^T A_j = \alpha_j^2 I
\]
for some $\alpha_j \geq 0$. Since $A$ is nonsingular, so is $A_j$, and hence $\alpha_j > 0$.
We now set $Q_j= (1/\alpha_j) A_j$ and $R_j=\alpha_j I$. Then $Q_j$ is orthogonal, $R_j$ is upper triangular with positive diagonal, and $Q_jR_j=A_j$. By uniqueness of the $QR$ factorization with positive diagonal, these are the factors computed in Theorem~\ref{thm:qr}. Both $Q_j$ and $R_j$ lie in $D$, so they commute with every $\rho_j'(g)$. Consequently, $Q = U\diag(Q_1,\dots,Q_r)U^T$ and $X = U\diag(R_1,\dots,R_r)U^T$ commute with every $\rho(g) = U\diag(\rho_1'(g),\dots,\rho_r'(g))U^T$, i.e., $Q, X \in \End_G(\rho)$. Finally, $U^TXU = \diag(\alpha_1I,\dots,\alpha_rI)$ is diagonal.
\end{proof}

The above condition is sufficient. However, it is not necessary (see Section~\ref{sec:examples}, where all multiplicities may exceed $1$, and Corollary~\ref{cor:qr-wedderburn}).

\subsection{Structure-preserving factorization in a Wedderburn-adapted basis}
We now show that both transformed-back factors can be chosen in the commutant in general, provided a basis finer than the isotypic one is given.

Let $D \in \{\R, \C, \HH\}$ and $\kappa = \dim_\R D$. We realize $D$ as real $\kappa \times \kappa$ matrices via its left-regular representation $L \colon D \to M_\kappa(\R)$ with respect to the standard basis $\{1\}$, $\{1, i\}$, or $\{1,i,j,k\}$, respectively. The map $L$ is an injective homomorphism of real algebras with $L(1) = I$, and $L(\bar d) = L(d)^T$ for all $d \in D$, since left multiplication by $d$ is $|d|$ times an isometry of $D \cong \R^\kappa$. For an $N \times N$ matrix $Y$ over $D$, its \emph{realification} $Y^\R \in M_{N\kappa}(\R)$ is obtained by replacing each entry $y$ of $Y$ by the block $L(y)$. Realification is a homomorphism of real algebras, and $(Y^*)^\R = (Y^\R)^T$, where $Y^*$ denotes the conjugate transpose. In particular, the realification of a unitary matrix over $D$ (i.e., $Y^*Y = I$) is orthogonal.

For each $i$ with $m_i > 0$, let $p_i = n_i/\kappa_i$. This is an integer, since $W_i$ is a vector space over the division algebra $D_i$ of dimension $p_i$. Let
\[
  \mathcal{S}_i = \bigl\{ (Y \otimes I_{p_i})^\R : Y \in M_{m_i}(D_i) \bigr\} \subseteq M_{d_i}(\R),
\]
where $Y \otimes I_{p_i}$ is the $m_ip_i \times m_ip_i$ matrix over $D_i$ with blocks $y_{kl} I_{p_i}$, and $d_i = m_i n_i = m_ip_i\kappa_i$.

\begin{definition}
Let $\rho$ be orthogonal. An orthogonal matrix $U \in O(n)$ is \emph{Wedderburn-adapted} for $\mathcal{C}(G,\rho)$ if
\[
  U^T\, \mathcal{C}(G,\rho)\, U = \mathcal{S}_{i_1} \oplus \cdots \oplus \mathcal{S}_{i_r},
\]
where $i_1 < \dots < i_r$ are the indices with $m_i > 0$.
\end{definition}

\begin{example}
Let $G=\Z/3\Z$ and let $R=\left(\begin{smallmatrix}-1/2&-\sqrt3/2\\ \sqrt3/2&-1/2\end{smallmatrix}\right)$ be the rotation by $2\pi/3$.
Let $\rho(1)=\diag(R,R)$ on $\R^4$. The rotation $R$ is a real irreducible representation of complex type, with
$n_i=\kappa_i=2$ and $p_i=1$, and it occurs in $\rho$ with multiplicity $m_i=2$; the trivial representation does not occur.
The $2\times2$ matrices commuting with $R$ are exactly $L(a+bi)=\left(\begin{smallmatrix}a&-b\\ b&a\end{smallmatrix}\right)$, so
\[
  \mathcal{C}(G,\rho)=\bigl\{Y^\R : Y\in M_2(\C)\bigr\}=\mathcal{S}_i,
\]
and $U=I_4$ is Wedderburn-adapted.
In contrast, $U'=\diag(1,1,1,-1)$ is also orthogonal, and its columns form an orthonormal basis of the (single) isotypic component.
But for $F=\diag(1,-1)$, the off-diagonal block of $U'^TAU'$ is $L(y_{12})F=\left(\begin{smallmatrix}a&b\\ b&-a\end{smallmatrix}\right)$,
which is not of the form $L(y)$ unless $y_{12}=0$. Hence $U'$ is not Wedderburn-adapted.
\end{example}

For the remainder of this subsection we assume that a Wedderburn-adapted orthogonal matrix $U$ is supplied as part of the input. Abstractly, the commutant on an isotypic component has the Wedderburn form $M_{m_i}(D_i)$, but identifying it with the specific real matrix model $\mathcal S_i$ by an orthogonal change of basis requires a choice compatible with the transpose involution. We do not prove the computability of such a choice here. In the first example of Section~\ref{sec:examples}, where $D_\pm=\R$ and $n_\pm=1$, the standard isotypic orthonormal basis is already Wedderburn-adapted.

\begin{corollary}
  Let $\rho$ be orthogonal. There exists a Type-2 machine which, given $\delta_\R^{n\times n}$-names of a Wedderburn-adapted orthogonal matrix $U$ for $\mathcal{C}(G,\rho)$ and of a nonsingular matrix $A \in \mathcal{C}(G,\rho)$, computes $\delta_\R^{n\times n}$-names of matrices $Q$, $X$ such that:
  \begin{enumerate}
    \item[(a)] $A = QX$;
    \item[(b)] $Q \in \mathcal{C}(G,\rho)$ and $Q \in O(n)$;
    \item[(c)] $X \in \mathcal{C}(G,\rho)$;
    \item[(d)] $U^TXU$ is upper triangular with positive diagonal entries.
  \end{enumerate}
  The matrices $Q$ and $X$ are uniquely determined by (a), (b) and (d).
  \label{cor:qr-wedderburn}
\end{corollary}

\begin{proof}
The machine computes the matrix $M = U^TAU$ and its classical $QR$ factorization $M = \tilde Q \tilde R$ with $\tilde R$ upper triangular with positive diagonal, by Gram--Schmidt orthogonalization of the columns of $M$. This is computable, since $M$ is nonsingular, so all norms by which we divide are nonzero reals. It outputs $Q = U\tilde QU^T$ and $X = U\tilde RU^T$. Then (a) holds, $Q$ is orthogonal, and $U^TXU = \tilde R$ satisfies (d). Uniqueness follows from the uniqueness of the $QR$ factorization of $M$ with positive diagonal.

It remains to show that $Q, X \in \mathcal{C}(G,\rho)$, i.e., that $\tilde Q$ and $\tilde R$ lie in $\bigoplus_i \mathcal{S}_i$. Since $U$ is Wedderburn-adapted, $M = \diag(M_{i_1},\dots,M_{i_r})$ with $M_i = (Y_i \otimes I_{p_i})^\R$ for some $Y_i \in M_{m_i}(D_i)$, and each $Y_i$ is invertible since $M$ is. The usual Gram--Schmidt construction also applies for $D_i=\HH$ when $D_i^{m_i}$ is regarded as a right $D_i$-module with its standard quaternionic Hermitian inner product. The diagonal entries of the resulting upper-triangular factor may be chosen to be positive real numbers. Using the standard inner product over $D_i$ and Gram--Schmidt with scalars acting from the right yields $Y_i = \hat Q_i \hat R_i$ with $\hat Q_i$ unitary and $\hat R_i$ upper triangular with real positive diagonal entries. Hence
\[
  M_i = \bigl((\hat Q_i \otimes I_{p_i})^\R\bigr)\bigl((\hat R_i \otimes I_{p_i})^\R\bigr),
\]
where the first factor is orthogonal, being the realification of a unitary matrix, and lies in $\mathcal{S}_i$. The second factor lies in $\mathcal{S}_i$ and is a real upper triangular matrix with positive diagonal: the entries of $\hat R_i \otimes I_{p_i}$ below the diagonal vanish, and each positive real diagonal entry $r$ becomes the block $rI_{\kappa_i}$. Assembling the blocks gives a $QR$ factorization of $M$ with positive diagonal whose factors lie in $\bigoplus_i \mathcal{S}_i$. By uniqueness, it equals $M = \tilde Q\tilde R$. Hence $\tilde Q, \tilde R \in \bigoplus_i \mathcal{S}_i = U^T\mathcal{C}(G,\rho)U$, i.e., $Q, X \in \mathcal{C}(G,\rho)$.
\end{proof}

Note that the machine of Corollary~\ref{cor:qr-wedderburn} needs no access to the group: all structural information is encoded in the Wedderburn-adapted matrix $U$.

\section{Examples}
\label{sec:examples}
\noindent\textbf{Centrosymmetric matrices of even order.}
Let $n = 2p$ and let $J$ be the $n \times n$ \emph{exchange matrix}. Matrices $A$ such that $AJ=JA$ are called \emph{centrosymmetric} \cite{Weaver1985}.
The matrix $J$ is a symmetric orthogonal involution, so it defines an orthogonal representation of $C_2 = \Z/2\Z$ on $\R^n$.
The eigenspaces of $J$ corresponding to the eigenvalues $+1$ and $-1$ are
\[
  V_+ = \{v : Jv=v \}, \quad V_- = \{v: Jv = -v\},
\]
both of dimension $p$. Since $J$ is symmetric, $V_+$ and $V_-$ are orthogonal. They are the isotypic components of this representation, with multiplicities $m_+ = m_- = p$.
Every centrosymmetric $A$ leaves $V_+$ and $V_-$ invariant: if $Jv = \pm v$, then $J(Av) = AJv = \pm Av$.
Let $\mathcal{U}=\{u_1,\dots,u_p\}$ be an orthonormal basis for $V_+$ and $\mathcal{V}=\{v_1,\dots,v_p\}$ an orthonormal basis for $V_-$.
We form $U$ by setting the first $p$ columns to be the elements of $\mathcal{U}$ and the last $p$ columns to be
the elements of $\mathcal{V}$. Since $V_+ \perp V_-$ and both bases are orthonormal, $U$ is orthogonal.
Now, $U$ block-diagonalizes every centrosymmetric $A$, i.e., $U^TAU = \diag(A_+, A_-)$ for every centrosymmetric $A$.
For nonsingular centrosymmetric $A$, we can use this to calculate ordinary $QR$ factorizations in the two isotypic blocks and then transform the factors back:
\begin{enumerate}
  \item Calculate $U^TAU=\diag(A_+, A_-)$.
  \item Calculate $A_+=Q_+ R_+$ and $A_-=Q_- R_-$.
  \item Calculate $Q=U \diag(Q_+,Q_-) U^T$ and $X=U \diag(R_+,R_-) U^T$.
\end{enumerate}
Note that both $Q$ and $X$ are centrosymmetric. Indeed, $J = U\diag(I_p, -I_p)U^T$, which commutes with $U\diag(Y_+,Y_-)U^T$ for any $p\times p$ blocks $Y_+$, $Y_-$. Thus, here the factors lie in the commutant although $m_\pm = p$ may exceed $1$, which shows that the condition of Proposition~\ref{prop:all-mi-equal-one} is not necessary. This is the special case of Corollary~\ref{cor:qr-wedderburn} in which $D_\pm = \R$ and $n_\pm = 1$, so that $U$ is Wedderburn-adapted. If $U$ is chosen to be the standard centrosymmetric block-diagonalizer, the factorization obtained here is the structure-preserving $QX$ factorization introduced by Burnik~\cite{Burnik2015}. Steele, Yalim and Welfert~\cite{steele2018} subsequently showed that the same factorization can be obtained from ordinary $QR$ factorizations of the two reduced blocks.
For odd order $n = 2p+1$, the same construction applies, with $\dim V_+ = p+1$ and $\dim V_- = p$.

\medskip
\noindent\textbf{Wedderburn-adapted example with $n_i>1$.}
Let $\pi$ be the permutation representation of the symmetric group $S_3$ on $\R^3$, and let $\rho=\pi\oplus\pi$ on $\R^6$. The real irreducible representations of $S_3$ are the trivial, the sign and the two-dimensional standard representation $\sigma$, all of real type. Since $\chi_\pi = \chi_{\mathrm{triv}} + \chi_\sigma$, we have $m_{\mathrm{triv}} = 2$, $m_{\mathrm{sign}} = 0$ and $m_\sigma = 2$. The commutant of $\pi$ is $\{\alpha I_3 + \beta\,zz^T : \alpha,\beta\in\R\}$, where $z=(1,1,1)^T$, and hence
\[
  \mathcal{C}(S_3,\rho)=\left\{\begin{pmatrix} A_{11} & A_{12}\\ A_{21} & A_{22}\end{pmatrix} : A_{kl}=\alpha_{kl}I_3+\beta_{kl}\,zz^T\right\}.
\]
Let $e=z/\sqrt3$, $u=(1,-1,0)^T/\sqrt2$ and $v=(1,1,-2)^T/\sqrt6$, and let
\[
  U=\left[\begin{pmatrix}e\\0\end{pmatrix},\begin{pmatrix}0\\e\end{pmatrix},\begin{pmatrix}u&v\\0&0\end{pmatrix},\begin{pmatrix}0&0\\u&v\end{pmatrix}\right]\in O(6).
\]
Using $e^Tzz^Te=3$ and $u^Tz=v^Tz=0$, one obtains
\[
  U^TAU=\bigl[\alpha_{kl}+3\beta_{kl}\bigr]_{k,l=1}^{2}\ \oplus\ \bigl(\bigl[\alpha_{kl}\bigr]_{k,l=1}^{2}\otimes I_2\bigr),
\]
so $U^T\mathcal{C}(S_3,\rho)U=\mathcal{S}_{\mathrm{triv}}\oplus\mathcal{S}_\sigma$ and $U$ is Wedderburn-adapted. For nonsingular $A$, let $[\alpha_{kl}+3\beta_{kl}]=Q_{\mathrm{t}}R_{\mathrm{t}}$ and $[\alpha_{kl}]=Q_\sigma R_\sigma$ be ordinary $QR$ factorizations of these $2\times 2$ matrices. The factorization of Corollary~\ref{cor:qr-wedderburn} is then
\[
  U^TQU = Q_{\mathrm{t}}\oplus(Q_\sigma\otimes I_2),\qquad
  U^TXU = R_{\mathrm{t}}\oplus(R_\sigma\otimes I_2),
\]
where $R_\sigma\otimes I_2=\left(\begin{smallmatrix} r_{11}I_2 & r_{12}I_2\\ 0 & r_{22}I_2\end{smallmatrix}\right)$ is upper triangular. Transforming back, the $(k,l)$ block of $Q$ is
\[
  Q_{kl}=(Q_\sigma)_{kl}\,I_3+\bigl((Q_{\mathrm{t}})_{kl}-(Q_\sigma)_{kl}\bigr)\,ee^T,
\]
using $uu^T+vv^T=I_3-ee^T$, and similarly for $X$. Thus both factors visibly have the form $\alpha I_3+\beta\,zz^T$ blockwise, i.e., they lie in $\mathcal{C}(S_3,\rho)$.

\section{Open Questions}
\label{sec:open-questions}
\begin{question}[generalizing to other groups]
  Can we generalize this result to commutants $\mathcal{C}(G, \rho)$
  where $G$ is:
  \begin{itemize}
    \item a finitely generated group?
    \item a compact group (e.g., can we apply Peter--Weyl theory for compact groups)?
    \item $\dots$
  \end{itemize}
\end{question}

\begin{question}[beyond commutants]
  Some well-known classes of matrices, such as symmetric or Toeplitz matrices, are not commutants.
  For what other matrix classes can we obtain similar computability results for UBD?
\end{question}

\begin{question}[Weihrauch degrees]
 Corollary~\ref{cor:RA} shows that, for the semisimple algebras $\R[A]$, universal block diagonalizers are not computable from the generator alone.
 For commutants of finite group representations, is the UBD still computable when less discrete information is given, for instance only a basis of the commutant, or only the matrices $\rho(s)$ without the group?
 If not, what are the Weihrauch degrees of the corresponding problems?
\end{question}

\section{Conclusion}
\label{sec:conclusion}
In the theory of Type-2 linear algebra, there is a general known obstruction in calculating
the rank of a matrix that is given as a Cauchy-named input. The reason is the
discontinuity of the rank. It is known that the rank is lower semicomputable
but not computable \cite[Theorem~7]{ZieglerBrattka2004}.

In this paper, we studied a particular case in which the rank obstruction is avoidable.
We found that finite-group structure recovers the discrete information required to
obtain the ranks of the corresponding projectors.

We have proved the computability of universal isotypic block diagonalization. Given the
finite group data as input, along with the representation of the finite group,
there is a Type-2 machine that (uniformly) computes the universal block diagonalizer
for the matrices from the commutant tied to this input information.

We applied this result to the case when the representation is orthogonal by computing ordinary $QR$ factorizations blockwise in the isotypic basis. After transforming back to the original coordinates we obtain a $QX$-type factorization relative to the computed isotypic basis. Given a Wedderburn-adapted orthogonal basis, the transformed-back factors can moreover be computed within the commutant.

The results of this paper are a concrete instance of a phenomenon of using
additional structure from the given input to obtain computability.
The general theory of using additional input structure to obtain
computability is future work.

\end{document}